\documentclass[a4paper,12pt]{article}
\usepackage{amssymb}
\usepackage{latexsym}
\usepackage{amsmath,amsthm,amssymb}
\usepackage{amsfonts}
\usepackage{eufrak}

\newtheorem{thm}{Theorem}[section]
\newtheorem{lem}[thm]{Lemma}
\newtheorem{pro}[thm]{Proposition}
\newtheorem{cor}[thm]{Corollary}

\theoremstyle{definition}
\newtheorem{defn}[thm]{Definition}
\newtheorem{exa}[thm]{Example}
\newtheorem{rem}[thm]{Remark}

\begin{document}

\begin{center}
{\Large On the Azuma inequality in spaces of subgaussian of rank $p$ random variables}
\end{center}
\begin{center}
{\sc Krzysztof Zajkowski}\\
Institute of Mathematics, University of Bialystok \\ 
Ciolkowskiego 1M, 15-245 Bialystok, Poland \\ 
kryza@math.uwb.edu.pl 
\end{center}

\begin{abstract}
For $p>1$ let a function $\varphi_p(x)=x^2/2$ if $|x|\le 1$ and $\varphi_p(x)=1/p|x|^p-1/p+1/2$ if $|x|>1$. 
For a random variable $\xi$ let $\tau_{\varphi_p}(\xi)$ denote
$\inf\{c\ge 0:\;\forall_{\lambda\in\mathbb{R}}\; \ln\mathbb{E}\exp(\lambda\xi)\le\varphi_p(c\lambda)\}$; $\tau_{\varphi_p}$ is a norm in a space $Sub_{\varphi_p}(\Omega)=\{\xi:\;\tau_{\varphi_p}(\xi)<\infty\}$ of $\varphi_p$-subgaussian random variables which we call {\it subgaussian of rank $p$ random variables}.
For $p=2$ we have the classic subgaussian random variables.

The Azuma inequality gives an estimate on the probability of the deviations of a zero-mean martingale $(\xi_n)_{n\ge0}$ with bounded increments from zero. In its classic form is assumed that $\xi_0=0$. In this paper it is shown a version of the Azuma inequality under assumption that $\xi_0$ is any  subgaussian of rank $p$ random variable. 
 
\end{abstract}

{\it 2010 Mathematics Subject Classification:} 60E15 

{\it Key words: 
Hoeffding-Azuma's inequality, $\varphi$-subgaussian random variables}
\section{Introduction}
Let $(\Omega,\mathcal{F},\mathbb{P})$ be a probability space. 
Hoeffding's inequality says that for bounded independent zero-mean random variables $\xi_1,...,\xi_n$ such that $\mathbb{P}(\xi_i\in[a_i,b_i])=1$, $i=1,...n$, the following estimate on the probability of the deviation of their sum $S_n=\sum_{i=1}^n\xi_i$ from zero 
$$
\mathbb{P}(|S_n|\ge\varepsilon)\le 2\exp\Big(-\frac{2\varepsilon}{\sum_{i=1}^n(b_i-a_i)^2}\Big)
$$
holds. The proof is based on Hoeffding's lemma: If $\xi$ is a random variable with mean zero such that $\mathbb{P}(\xi\in[a,b])=1$ then
\begin{equation}
\label{est1}
\mathbb{E}\exp(\lambda\xi)\le \exp\Big(\frac{1}{8}(b-a)^2\lambda^2\Big);
\end{equation}
compare \cite{Hoeff}.

Let us observe that the above inequality means that the centered bounded random variable $\xi$ is a subgaussian random variable. Let us recall the notion of the subgaussian random variable which was introduced by Kahane in \cite{Kahane}. A random variable $\xi$ is {\it subgaussian} if there
exists a number $c\in[0,\infty)$ such that for every $\lambda\in\mathbb{R}$ the following inequality holds
$$
\mathbb{E}\exp(\lambda\xi)\le \exp\Big(\frac{c^2\lambda^2}{2}\Big),
$$
that is the moment generating function of $\xi$ is majorized by the moment generating function of some centered gaussian random variables with variance $c^2$ 
(see Buldygin and Kozachenko \cite{BK} or \cite[Ch.1]{BulKoz}). In terms of the cumulant generating functions this condition takes a form:
$\ln\mathbb{E}\exp(\lambda \xi)\leq c^2\lambda^2/2$. 

For a random variable $\xi$ a number $\tau(\xi)$ defined as follows
$$
\tau(\xi)=\inf\Big\{c\ge 0:\;\forall_{\lambda\in\mathbb{R}}\; \ln\mathbb{E}\exp(\lambda\xi)\le\frac{c^2\lambda^2}{2}\Big\}
$$
is a norm in a space $Sub(\Omega)=\{\xi\in L(\Omega):\;\tau(\xi)<\infty\}$ of subgaussian random variables, where $L(\Omega)$ denote the family of all real valued random variables defined on $\Omega$.  The space $Sub(\Omega)$ is a Banach space with respect to the norm $\tau$;
see \cite[Ch.1, Th.1.2]{BulKoz}. 

Immediately by the definition of $\tau$ we have that 
$$
\ln\mathbb{E}\exp(\lambda\xi)\le\frac{\tau(\xi)^2\lambda^2}{2}.
$$
By using this inequality one can show some estimate of tails distribution of $\xi$ in the form
$$
\mathbb{P}(|\xi|\ge\varepsilon)\le 2\exp\Big(-\frac{\varepsilon^2}{2\tau(\xi)^2}\Big);
$$ 
see \cite[Ch.1, Lem.1.3]{BulKoz}. If we do not know a value of the norm $\tau(\xi)$ but know some upper bound on it then we can formulate the above inequality substituting this bound instead of $\tau(\xi)$. And so for a centered  random variable $\xi$ essentially bounded by the interval $[a,b]$, by virtue of the inequality (\ref{est1}), we have that
$\tau(\xi)\le (b-a)/2$ and we get Hoeffding's inequality for one variable. Because for a sum of independent subgaussian random variables the following 
follows
$$
\tau\Big(\sum_{i=1}^n\xi_i\Big)^2\le  \sum_{i=1}^n\tau(\xi_i)^2,
$$
see \cite[Ch.1, Lem.1.7]{BulKoz}, then combining these facts we obtain the proof of Hoeffding's inequality; compare \cite[Ch.1, Cor.1.3]{BulKoz}.

\section{Spaces of subgaussian of rank $p$ random variables}
One can generalize the notion of subgaussian random variables to  classes of $\varphi$-subgaussian r.v.s (see \cite[Ch.2]{BulKoz}). A continuous even convex function $\varphi(x)$ $(x\in \mathbb{R}$) is called a {\em $N$-function}, if the following condition hold:\\
(a) $\varphi(0)=0$ and $\varphi(x)$ is monotone increasing for $x>0$,\\
(b) $\lim_{x\to 0}\varphi(x)/x=0$ and $\lim_{x\to \infty}\varphi(x)/x=\infty$.\\
It is called a {\it quadratic $N$-function}, if in addition $\varphi(x)=ax^2$ for all $|x|\le x_0$, with $a>0$ and $x_0>0$. The  quadratic condition is needed to ensure nontriviality for classes of $\varphi$-subgaussian random variables (see \cite[Ch.2, p.67]{BulKoz}).

Let $\varphi$ be a quadratic $N$-function. A random variable $\xi$ is said to be {\it $\varphi$-subgaussian} if there is a constant $c>0$ such that 
$\ln\mathbb{E}\exp(\lambda \xi)\leq \varphi(c\lambda)$. The {\it $\varphi$-subgaussian standard (norm) $\tau_{\varphi}(\xi)$} is defined as 
$$
\tau_{\varphi}(\xi)=\inf\{c\ge 0:\;\forall_{\lambda\in\mathbb{R}}\; \ln\mathbb{E}\exp(\lambda \xi)\le\varphi(c\lambda)\};
$$
a space $Sub_\varphi(\Omega)=\{\xi\in L(\Omega):\;\tau_{\varphi}(\xi)<\infty\}$ with the norm $\tau_{\varphi}$ is a Banach space (see \cite[Ch.2, Th.4.1]{BulKoz}). 

Now we define some class of such spaces.
\begin{defn}
Let for $p > 1$
$$
\varphi_p(x)=\left\{
\begin{array}{ccl}
\frac{x^2}{2}, & {\rm if} & |x|\le 1,\\
\frac{1}{p}|x|^p-\frac{1}{p}+\frac{1}{2}, & {\rm if} & |x|>1.
\end{array}
\right.
$$
\end{defn} 
The functions $\varphi_p$ are examples of quadratic $N$-functions. 
Let us observe that if $1< p'< p$ then $\varphi_{p'}\le \varphi_p$ and, in consequence,  $Sub_{\varphi_{p'}}(\Omega)\subset Sub_{\varphi_p}(\Omega)$, since 
$\tau_{\varphi_{p'}}(\xi)\ge \tau_{\varphi_p}(\xi)$ for any $\xi\in Sub_{\varphi_{p'}}(\Omega)$. Let us emphasize that the spaces $\{Sub_{\varphi_p}(\Omega):\;p>1\}$ form
increasing family with respect to $p$. 

For the  sake of completeness our presentation we show that any centered bounded random variable is subgaussian of any rank $p$. In general it is $\varphi$-subgaussian random variable for any quadratic $N$-function $\varphi$ (see \cite[Ex.3.1]{Rita}).   
\begin{pro}
Let $\varphi$ be an $N$-function such that $\varphi(x)=x^2/2$ for $|x|\le 1$. If $\xi$ is a bounded random variables with $\mathbb{E}\xi=0$ then $\xi\in Sub_\varphi(\Omega)$.
\end{pro}
\begin{proof}
If $\xi = 0$ with probability one then
$$
0=\psi_\xi(\lambda)\le \varphi(c\lambda)
$$
for $\lambda\in\mathbb{R}$, $c\ge 0$ and any $N$-function $\varphi$. Let us recall that if $\xi$ is bounded but nonconstant then $\psi_\xi$ is strictly convex on $\mathbb{R}$ and it follows positivity $\psi_\xi^{\prime\prime}$ on whole $\mathbb{R}$.

Let $|\xi|\le d$ almost surely then 
$$
\psi_\xi(\lambda)\le d|\lambda|=\varphi(\varphi^{(-1)}(d|\lambda|))\le\varphi(d\lambda\varphi^{(-1)}(1))
$$
for $|\lambda|> 1/d$ (see \cite[Ch.2, Lem. 2.3]{BulKoz}), where $\varphi^{(-1)}(x)$, $x\ge 0$, is the inverse function of $\varphi(x)$, $x\ge 0$. For $|\lambda|\le 1$,
by the Taylor theorem, we get
$$
\psi_\xi(\lambda)=\psi_\xi(0)+\psi_\xi'(0)\lambda+\frac{1}{2}\psi_\xi^{\prime\prime}(\theta)\lambda^2=\frac{1}{2}\psi_\xi^{\prime\prime}(\theta)\lambda^2
$$
for some $\theta$ between $0$ and $\lambda$. Let $c=\max_{\lambda\in[-1,1]}\psi_\xi^{\prime\prime}(\lambda)$ then
$
\psi_\xi(\lambda)\le 1/2(\sqrt{c}\lambda)^2
$.
Let us observe that for $|\lambda|\le 1/\sqrt{c}$ the function $1/2(\sqrt{c}\lambda)^2=\varphi(\sqrt{c}\lambda)$. Without loss of generality we may assume that $d>\sqrt{c}$ that is $1/d<1/\sqrt{c}$.   Taking $b=\max\{d\varphi^{(-1)}(1),\;\sqrt{c}\}$ we get
$$
\psi_\xi(\lambda)\le \varphi(b\lambda)
$$
for every $\lambda\in\mathbb{R}$ which follows $\xi\in Sub_\varphi(\Omega)$.
\end{proof}

Let $\varphi(x)$ ($x\in\mathbb{R}$) be a real-valued function. The function $\varphi^\ast(y)$ ($y\in\mathbb{R}$) defined by $\varphi^\ast(y)=\sup_{x\in\mathbb{R}}\{xy-\varphi(x)\}$ is called the {\it Young-Fenchel transform}
or the {\it convex conjugate} of $\varphi$ (in general, $\varphi^\ast$ may take value $\infty$). It is known that if $\varphi$ is a quadratic $N$-function then $\varphi^\ast$ is quadratic $N$-function too. For instance, since our $\varphi_p$ is a differentiable (even at $\pm 1$) function one can easy check that $\varphi_p^\ast=\varphi_q$ for $p,q>1$, if $1/p+1/q=1$.


An exponential estimate for tails distribution of a random variable $\xi$ belonging to the space $Sub_{\varphi}(\Omega)$ is as follows:
\begin{equation}
\label{estm}
\mathbb{P}(|\xi|\ge\varepsilon)\le 2\exp\Big(-\varphi^\ast\Big(\frac{\varepsilon}{\tau_{\varphi}(\xi)}\Big)\Big);
\end{equation}
see \cite[Ch.2, Lem.4.3]{BulKoz}. Moreover $\xi\in Sub_{\varphi}(\Omega)$ if and only if $\mathbb{E}\xi=0$ and there exist constant $C>0$ and $D>0$ such that
$$
\mathbb{P}(|\xi|\ge\varepsilon)\le C\exp\Big(-\varphi^\ast\Big(\frac{\varepsilon}{D}\Big)\Big)
$$
for every $\varepsilon>0$; compare \cite[Ch.2, Cor.4.1]{BulKoz}.

By virtue of the above facts one can show examples of subgaussian of rank $p$ random variables.
\begin{exa}
Let $q>1$ and a random variable $\xi$ has the double Weibull distribution with a density function
$$
g_\xi(x)=\frac{1}{2}|x|^{q-1}\exp\Big\{-\frac{1}{q}|x|^q\Big\}.
$$
Then $\xi\in Sub_{\varphi_p}(\Omega)$, where $1/p+1/q=1$.
\end{exa}
\begin{exa}
Let $\xi\sim\mathcal{N}(0,1)$, then $\eta=|\xi|^{2/q}-\mathbb{E}|\xi|^{2/q}\in Sub_{\varphi_p}(\Omega)$.
\end{exa}

\section{On the Azuma inequality}
Before  we prove some general form of Azuma's inequality, first we show some upper bound on the norm of centered bounded random variable in $Sub_{\varphi_p}(\Omega)$ for any $p>1$.
\begin{lem}
\label{lem1}
Let $\xi$ be a bounded random variable such that $\mathbb{P}(\xi\in [a,b])=1$ and $\mathbb{E}\xi=0$. Let $c$ denote the number $(b-a)/2$  and  $d$ the number $\max\{-a,b\}$. Then for every $p>1$ 
the norm 
$\tau_{\varphi_p}(\xi)\le \gamma_r=c^2/(2d)\{r[2(d/c)^2+1/r-1/2)]\}^{1/r}$, where $r=\min\{p,2\}$.
\end{lem}
\begin{proof}
If $p\ge2$ then $r=2$ and $\gamma_r=\gamma_2=c$. By Hoeffding's lemma we have that $\tau_{\varphi_2}(\xi)\le c$, and because for every $p\ge 2$ the norm $\tau_{\varphi_p}(\xi)\le \tau_{\varphi_2}(\xi)$ then the Lemma follows.

Assume now that $1<p<2$, then $r=p$.
Let us note that there exist a very simple estimate of the cumulant generating function of $\xi$:
$$
\psi_\xi(\lambda)=\ln\mathbb{E}\exp(\lambda\xi)\le \ln\mathbb{E}\exp(d|\lambda|)=d|\lambda|.
$$
We can form some majorant of $\psi_\xi$ with this function and Hoeffding's bound.

Solving the equation  $d\lambda=c^2\lambda^2/2$ we obtain $\lambda=2d/c^2$. Let us emphasize that for $1<p<2$ a function
$$
f(\lambda):=\min\Big\{\frac{c^2\lambda^2}{2},d|\lambda|\Big\}=\left\{
\begin{array}{ccl}
\frac{c^2\lambda^2}{2}, & {\rm if} & |\lambda|\le \frac{2d}{c^2} \\
d|\lambda|, & {\rm if} & |\lambda|>\frac{2d}{c^2}
\end{array}
\right.
$$ 
is a majorant of $\psi_\xi$. Let us observe now that $f(2d/c^2)=2(d/c)^2\ge 2$, since $d\ge c$. We find now $\gamma_p$ such that $\varphi_p(\gamma_p2d/c^2)=2(d/c)^2$. Notice that $\varphi_p([-1,1])=[0,1/2]$. It follows that solving the equation $\varphi_p(\gamma_p2d/c^2)=2(d/c)^2$ we should use the form of $\varphi_p(x)$ for $|x|>1$, i.e. $\varphi_p(x)=1/p|x|^p-1/p+1/2$. A solution of the equation
$$
\frac{1}{p}\Big|\gamma_p\frac{2}{d}\Big|^\frac{1}{p}-\frac{1}{p}+\frac{1}{2}=2\Big(\frac{d}{c}\Big)^2
$$
has the form 
$$
\gamma_p=\frac{c^2}{2d}\Big\{p\Big[2\Big(\frac{d}{c}\Big)^2+\frac{1}{p}-\frac{1}{2}\Big]\Big\}^\frac{1}{p}.
$$
Let us emphasize that for $p\in(1,2]$ we have the following inequality
$$
f(\lambda)\le \varphi_p(\gamma_p\lambda).
$$
Since $f$ is the majorant of $\psi_\xi$, we get that
$$
\psi_\xi(\lambda)\le \varphi_p(\gamma_p\lambda)
$$
for every $\lambda\in\mathbb{R}$. Thus, by definition of the norm $\tau_{\varphi_p}$ at $\xi$, 
$$
\tau_{\varphi_p}(\xi)\le \gamma_r=\frac{c^2}{2d}\Big\{r\Big[2\Big(\frac{d}{c}\Big)^2+\frac{1}{r}-\frac{1}{2}\Big]\Big\}^\frac{1}{r},
$$
in the case $r=p$, which completes the proof.
\end{proof}
\begin{rem}
\label{uwg}
Let us note once again that $\gamma_2=c$. Because $\gamma_p$ is the solution of the equation  $\varphi_p(\gamma_p2d/c^2)=2(d/c)^2$ and $\varphi_p\ge\varphi_{p'}$ for $p\ge p'$
then $\gamma_{p}\le \gamma_{p'}$. More precisely $\gamma_p$ is strictly increasing as $p$ is decreasing to $1$. Thus there exists  $\lim_{p\searrow 1}\gamma_p$. Denote it by $\gamma_1$. One can check that $\gamma_1=d+c^2/4d$.  
\end{rem}
Now we can formulate our main result.
\begin{thm}
\label{tw1}
Let $\xi_0$ be a subgaussian of a rank $p$ random variable  such that  $\tau_{\varphi_p}(\xi_0)\le d_0$ and $(\xi_n)_{n\ge0}$ be a martingale with bounded increments, i.e. $|\xi_n-\xi_{n-1}|\le d_n$ almost surely for $n=1,2,...$. Let  $c$ denote $\sqrt{\sum_{i=1}^nd_i^2}$ and $d$ the number $\sum_{i=1}^nd_i$. Then
$$
\mathbb{P}(|\xi_n|\ge\varepsilon)\le 2\exp\Big(-\varphi_q\Big(\frac{\varepsilon}{\sqrt[r]{\gamma_r^r+d_0^r}}\Big)\Big),
$$
where $1/p+1/q=1$, $r:=\min\{p,2\}$ and $\gamma_r=c^2/(2d)\{r[2(d/c)^2+1/r-1/2)]\}^{1/r}$.
\end{thm}
\begin{proof}
First we recall an argument which gives the similar estimate on the moment generating function of martingales with bounded increments as in the  case of sums of independent random variables. 
For the martingale $(\xi_n)_{n\ge 0}$ we have  
$$
\mathbb{E}\exp(\lambda\xi_n)=\mathbb{E}\Big(\exp(\lambda\xi_{n-1})\mathbb{E}\big(\exp(\lambda(\xi_n-\xi_{n-1}))\big|\mathcal{F}_{n-1} \big)\Big),
$$
where $\mathcal{F}_{n-1}$ denotes $\sigma$-field generated by random variables $\xi_0,\xi_1,...,\xi_{n-1}$.

Now we find a bound for $\mathbb{E}(\exp(\lambda(\xi_n-\xi_{n-1}))|\mathcal{F}_{n-1})$. Let $\eta_n:=(\xi_n-\xi_{n-1})/d_n$. Observe that 
$-1\le\eta_n\le 1$ a.s.. By convexity of the natural exponential function we get
$$
\exp(\lambda(\xi_n-\xi_{n-1}))=\exp(d_n\lambda\eta_n)\le \frac{1+\eta_n}{2}\exp(d_n\lambda)+\frac{1-\eta_n}{2}\exp(-d_n\lambda),
$$ 
and, in consequence,
$$
\mathbb{E}\big(\exp(d_n\lambda\eta_n)\big|\mathcal{F}_{n-1}\big)\le \frac{1}{2}\exp(d_n\lambda)+\frac{1}{2}\exp(-d_n\lambda),
$$
since $\mathbb{E}(\eta_n|\mathcal{F}_{n-1})=0$. By virtue of the inequality $1/2\exp(d_n\lambda)+1/2\exp(-d_n\lambda)\le\exp(\lambda^2d_n^2/2)$ one gets
$$
\mathbb{E}\exp(\lambda\xi_n)\le\exp\Big(\frac{\lambda^2d_n^2}{2}\Big)\mathbb{E}\exp(\lambda\xi_{n-1})
$$
and, inductively, 
$$
\mathbb{E}\exp(\lambda\xi_n)\le\exp\Big(\frac{\lambda^2\sum_{i=1}^nd_i^2}{2}\Big)\mathbb{E}\exp(\lambda\xi_0).
$$
Taking the logarithm of both sides we obtain
$$
\ln\mathbb{E}(\exp(\lambda\xi_n)\le\frac{\lambda^2\sum_{i=1}^nd_i^2}{2}+\ln\mathbb{E}\exp(\lambda\xi_0),
$$
that is
\begin{equation}
\label{row1}
\psi_{\xi_n}(\lambda)\le \varphi_2\Big(\lambda\Big(\sum_{i=1}^nd_i^2\Big)^{1/2}\Big)+\psi_{\xi_0}(\lambda).
\end{equation}
Because  $\psi_{\xi_0}(\lambda)\le \varphi_p(d_0\lambda)$ and a random variable $\xi_n-\xi_0$ is the bounded random variable ($|\xi_n-\xi_0|\le\sum_{i=1}^nd_i$ a.s.), by Lemma \ref{lem1}, 
we can rewrite the above estimate on  $\psi_{\xi_n}$ as follows
\begin{equation}
\label{psin}
\psi_{\xi_n}(\lambda)\le \varphi_p(\gamma_r\lambda)+\varphi_p(d_0\lambda), 
\end{equation}
where $\gamma_r$ is as in Lemma \ref{lem1} ($c=\sqrt{\sum_{i=1}^nd_i^2}$ and $d=\sum_{i=1}^nd_i$).

The composition $\varphi_p$ with the function $\sqrt[r]{\cdot}$ is still convex. By properties of $N$-functions 
(see \cite[Ch.2, Lem.2.2]{BulKoz}) for $\lambda>0$ we get 
\begin{eqnarray*}
\varphi_p(\gamma_r\lambda)+\varphi_p(d_0\lambda)
&=&\varphi_p\big(\sqrt[r]{\gamma_r^r\lambda^r}\big)+\varphi_p\big(\sqrt[r]{d_0^r\lambda^r}\big) \\
\; &\le& \varphi_p\Big(\lambda\sqrt[r]{\gamma_r^r+d_0^r}\Big),
\end{eqnarray*}
which combining with (\ref{psin}) gives
$$
\psi_{\xi_n}(\lambda)\le \varphi_p\Big(\lambda\sqrt[r]{\gamma_r^r+d_0^r}\Big).
$$
Because $\varphi_p$ is the even function then the above inequality is valid for any $\lambda$. It means that the random variable $\xi_n\in Sub_{\varphi_p}(\Omega)$ and its norm $\tau_{\varphi_p}(\xi_n)\le \sqrt[r]{\gamma_r^r+d_0^r}$. 

Recall that the convex conjugate $\varphi_p^\ast=\varphi_q$, where $1/p+1/q=1$. By (\ref{estm}) and the above estimate of $\tau_{\varphi_p}(\xi_n)$  we obtain our 
inequality.
\end{proof}

\begin{rem}
 For $\xi_0=0$ a.s. $d_0=0$ and we can assume that $\xi_0$ is subgaussian of rank $2$ (classic subgaussian). Recall that if $p=2$ then $q=2$ and $\varphi_q(x)=x^2/2$. In this case $\gamma_r=\gamma_2=c$ and  we get the classic form of Hoeffding-Azuma's inequality.
\end{rem} 
Let us observe that $\xi_0=0$ a.s. is subgaussian of any rank $p$. Consider more precisely case $1<p<2$. If  $1<p<2$ then the H\"older conjugate $q>2$. Let us recall that for $\varepsilon\in\mathbb{R}$ $\varphi_q(\varepsilon)\ge \varphi_2(\varepsilon)$ and moreover $\varphi_q(\varepsilon)> \varphi_2(\varepsilon)$ if
$|\varepsilon|>1$. Because $c=\gamma_2<\gamma_p$, where $c=\sqrt{\sum_{i=1}^n d_i^2}$, there exists exactly one $\varepsilon_p>0$ such that
$$
\varphi_q\Big(\frac{\varepsilon_p}{\gamma_p}\Big)=\varphi_2\Big(\frac{\varepsilon_p}{c}\Big),
$$
i.e. $\varepsilon_p$ is the unique solution of the equation
$$
\frac{1}{q}\Big(\frac{\varepsilon_p}{\gamma_p}\Big)^q-\frac{1}{q}+\frac{1}{2}=\frac{\varepsilon_p^2}{2c}.
$$
If $|\varepsilon|<\varepsilon_p$ then $\varphi_q(\varepsilon/\gamma_p)<\varphi_c(\varepsilon/c)$ and $\varphi_q(\varepsilon/\gamma_p)>\varphi_c(\varepsilon/c)$ if
$|\varepsilon|>\varepsilon_p$. It follows that for $|\varepsilon|>\varepsilon_p$ the estimate
$$
\mathbb{P}(|\xi_n|\ge\varepsilon)\le 2\exp\Big(-\varphi_q\Big(\frac{\varepsilon}{\gamma_p}\Big)\Big)
$$
is sharper than the Hoeffding-Azuma's inequality.
It means that by $\xi_0=0$ Theorem \ref{tw1} is some supplement of the Hoeffding-Azuma inequality. Moreover in this Theorem we considered  the case when $\xi_0$ is any subgaussian of rank $p$ random variable.

Many another examples of concentration inequalities one can  find for instance in \cite{McDi}. Let us emphasize that most of them  concern  the case of independent summands. The Azuma inequality is dealt with dependent ones. Let us note that in the original Azuma's paper \cite{Azuma} are considered bounded increments satisfying some general conditions that hold for martingales increments. It is  most important to us that we can find  some bound of the norm of their sums in the spaces of subgaussian of rank $p$ random variables which allow us to get an estimate for the probabilities of tail distributions. Applications of such estimates may by multiple. I would like to drew attention on some application  to prove of the strong laws of large numbers for dependent random variables in these spaces
(see \cite{Zaj}).


\begin{thebibliography}{                    }

\bibitem{Azuma}
K. Azuma, {\it Weighted sums of certain dependent random variables}, Tokohu Mathematical Journal 19 (1967), 357-367.


\bibitem{BulKoz}
V. Buldygin, Yu. Kozachenko, {\it Metric Characterization of Random Variables and Random Processes}, Amer.Math.Soc., Providence, RI, 2000. 

\bibitem{BK}
V. Buldygin, Yu. Kozachenko, {\it Subgaussian random variables}, Ukrainian Math. J. 32 (1980), 483-489.

\bibitem{Rita}
R. Giuliano Antonini, Yu. Kozaczenko, T. Nikitina, {\it Spaces of $\varphi$-sub-Gaussian random variables}, Rend. Accad. Naz. Sci. XL Mem. Mat. Appl. 27(5), (2003), 95-124.



\bibitem{Hoeff}
W. Hoeffding, {\it Probability for sums of bounded random variables}, Journal of the American Statistical Association 58 (1963), 13-30.

\bibitem{Kahane}
J.P. Kahane, {\it Local properties of functions in terms of random Fourier series (French)}, Stud. Math., 19 (no. 1), 1-25 (1960).



\bibitem{McDi}
C. McDiarmid, {\it Concentration}, in Probabilistic Methods for Algorithmic Discrete Mathematics, 1998, 195-248. 


\bibitem{Zaj}
K. Zajkowski, {\it On the strong law of large numbers for $\varphi$-subgaussian random variables}, arXiv:1607.03035.

\end{thebibliography}
\end{document}